\documentclass[12pt,english,a4paper]{smfart}

\usepackage[T1]{fontenc}
\usepackage{lmodern}
\usepackage{smfthm}
\usepackage[headings]{fullpage}

\newcommand{\Qp}{\mathbf{Q}_p}
\newcommand{\Zp}{\mathbf{Z}_p}
\newcommand{\ZZ}{\mathbf{Z}}
\renewcommand{\phi}{\varphi}
\renewcommand{\projlim}{\varprojlim}
\newcommand{\Gal}{\mathrm{Gal}}
\newcommand{\Iw}{\mathrm{Iw}}
\newcommand{\Fil}{\mathrm{Fil}}
\newcommand{\Nm}{\mathrm{N}}
\newcommand{\alg}{\mathrm{alg}}
\newcommand{\cyc}{\mathrm{cyc}}

\newcommand{\dfont}{\mathrm{D}}
\newcommand{\pr}{\operatorname{pr}}
\newcommand{\dcroc}[1]{[\![ #1 ]\!]}

\author{Laurent Berger}
\address{UMPA ENS de Lyon \\
UMR 5669 du CNRS}
\email{laurent.berger@ens-lyon.fr}
\urladdr{perso.ens-lyon.fr/laurent.berger/}

\title{Universal norms and de Rham representations}

\begin{document}

\begin{abstract}
The purpose of this note is to show that the computation of the universal norms in Iwasawa cohomology for the $p$-adic cyclotomic extension implies an analogous result for any $p$-adic extension that contains the cyclotomic extension. In addition, we prove a twisted cyclotomic variant, which applies to all Lubin--Tate extensions.
\end{abstract}

\date{\today}

\thanks{The author's research was partially supported by the ANR grant ANR-25-CE40-4664}

\maketitle

\setlength{\baselineskip}{17pt}

\section{Definitions and main result}

Let $K$ be a finite extension of $\Qp$ (or, more generally, a finite extension of $W(k)[1/p]$, where $k$ is a perfect field of characteristic $p$). Let $G_K = \Gal(K^{\alg}/K)$ denote the absolute Galois group of $K$ and let $V$ be a $p$-adic representation of $G_K$. Let $K_\infty \subset K^{\alg}$ be an infinite extension of $K$. We consider the Iwasawa cohomology group $H^1_{\Iw}(K_\infty/K,V) = \Qp \otimes_{\Zp} \projlim_L H^1(L,T)$ where $T$ is a lattice of $V$, $L$ runs through the set $E(K_\infty/K)$ of finite subextensions of $K_\infty/K$, and the transition maps are the corestriction maps. If $V$ is a de Rham representation, we consider the submodule $H^1_{\Iw,g}(K_\infty/K,V)$ of $H^1_{\Iw}(K_\infty/K,V)$ consisting of those classes whose image in $H^1(L,V)$ is in the subspace $H^1_g(L,V)$ of $H^1(L,V)$ parametrizing de Rham extensions, for every $L \in E(K_\infty/K)$. 

Let $\Fil^1 V$ be the largest subrepresentation of $V$ whose Hodge--Tate weights are all $\geq 1$. The space $H^1_{\Iw}(K_\infty/K,\Fil^1 V)$ is then contained in $H^1_{\Iw,g}(K_\infty/K,V)$. When $K_\infty = K_{\cyc} := K( \mu_{p^\infty} )$ is the cyclotomic extension of $K$, the main result of \cite{Ber05} is that this inclusion is almost an equality. More precisely, Theorem A of ibid is the following 
(where $H_K = \Gal(K^{\alg}/K_{\cyc})$).

\begin{theo}
\label{theoA}
Let $V$ be a de Rham representation of $G_K$ and let $K_{\cyc}/K$ be the cyclotomic extension.
\begin{enumerate}
\item If $V$ has no subquotient fixed by $H_K$, then $H^1_{\Iw,g}(K_{\cyc}/K,V) = H^1_{\Iw}(K_{\cyc}/K,\Fil^1 V) $
\item In general, $H^1_{\Iw}(K_{\cyc}/K,\Fil^1 V) \subset H^1_{\Iw,g}(K_{\cyc}/K,V)$ and the quotient is a finite-dimensional $\Qp$-vector space.
\end{enumerate}
\end{theo}

This result proved a conjecture of Nekov\'a\v{r} (conjecture A of \cite{PR00}), building on earlier work of Perrin-Riou \cite{PR00} who used her big exponential map. The proof of Theorem \ref{theoA} uses the theory of $(\phi,\Gamma)$-modules. For other extensions $K_\infty/K$, in the absence of a good theory of $(\phi,\Gamma)$-modules, proving an analogue of this result seems difficult. There has however been recent encouraging progress by Ponsinet \cite{Pon24,Pon25}, using the Fargues--Fontaine curve instead of $(\phi,\Gamma)$-modules. 

The purpose of this note is to show that the methods used to prove Theorem \ref{theoA} also imply the following result.

\begin{theo}
\label{theoB}
If $V$ is a de Rham representation of $G_K$ and if $K_\infty$ \emph{contains} the cyclotomic extension $K_{\cyc}$ of $K$, then $H^1_{\Iw,g}(K_\infty/K,V)  / H^1_{\Iw}(K_\infty/K,\Fil^1 V)$ is a finite-dimensional $\Qp$-vector space.
\end{theo}


\section{Proof of the theorem}
\label{secproof}


The main result that we need in order to prove Theorem \ref{theoB} is the following (Prop II.3.1 of \cite{Ber05}; recall that $H^1_{\Iw}(K_\cyc/K,V) = \dfont^\dagger(V)^{\psi=1}$ in the notation of ibid).

\begin{prop}
\label{prop1}
If $k \geq 0$ and if $V$ is a de Rham representation of $G_K$ whose Hodge--Tate weights are all $\geq -k$, and which has no subrepresentation all of whose Hodge--Tate weights are $\geq 1-k$, then $H^1_{\Iw,g}(K_{\cyc}/K,V) \subset V^{H_K}$.
\end{prop}

Recall that $K_\infty$ is now assumed to contain the cyclotomic extension $K_{\cyc}$ of $K$.

\begin{coro}
\label{coro1}
If $k \geq 0$ and if $V$ is a de Rham representation of $G_K$ whose Hodge--Tate weights are all $\geq -k$, and which has no subrepresentation all of whose Hodge--Tate weights are $\geq 1-k$, then $H^1_{\Iw,g}(K_\infty/K,V)$ is a finite-dimensional $\Qp$-vector space of dimension $\leq \dim(V)$.
\end{coro}

\begin{proof}
We have a map $H^1_{\Iw,g}(K_\infty/K,V) \to \projlim_{L \in E(K_\infty/K)} H^1_{\Iw,g}(L_{\cyc}/L,V)$, which is injective as each $H^1_{\Iw}(L_{\cyc}/L,T)$ is $\Zp$-torsion free. By Proposition \ref{prop1}, each of the $H^1_{\Iw,g}(L_{\cyc}/L,V)$ is a finite-dimensional $\Qp$-vector space of dimension $\leq \dim(V)$. 
This implies the result by Lemma \ref{lemm1} below.
\end{proof}

\begin{lemm}
\label{lemm1}
If $d \in \ZZ_{\geq 0}$ and $\{ V_i \}_{i \in I }$ is a directed inverse system 
of $\Qp$-vector spaces with $\dim V_i \leq d$ for all $i \in I$, then $\dim \projlim_{i \in I} V_i \leq d$.
\end{lemm}

\begin{proof}
Take $c_1, \hdots, c_{d+1} \in \projlim_{i \in I} V_i$. Let $R_i \subset \Qp^{d+1}$ be the set of $(a_1,\hdots,a_{d+1})$ such that $\sum_{n=1}^{d+1} a_n  \cdot \pr_{V_i}(c_n) = 0$ in $V_i$. Each $R_i$ is nonzero since $\dim V_i \leq d$ and by directedness, each finite intersection of $R_i$ is nonzero. Since $\Qp^{d+1}$ is finite-dimensional, the intersection of the $R_i$  is nonzero. This gives a nontrivial linear relation between the $c_n$. 
\end{proof}

Theorem \ref{theoB} now follows from the same d\'evissage arguments as in \cite{Ber05}, which we now give. Let $\Fil^j V$ denote the largest subrepresentation of $V$ whose Hodge--Tate weights are all $\geq j$. Note that if $L/K$ is a finitely ramified extension and $V$ is a representation of $G_K$, then the largest $G_L$-subrepresentation $\Fil_L^j V$ of $V$ whose Hodge--Tate weights are all $\geq j$ is stable under $G_K$. Hence it is equal to $\Fil_K^j V$, so we can drop the subscript from the notation. Now take $k \geq 0$. We have an exact sequence
\[ 0 \to H^1_{\Iw,g}(K_\infty/K,\Fil^{1-k} V) \to H^1_{\Iw,g}(K_\infty/K, \Fil^{-k} V) \to H^1_{\Iw,g}(K_\infty/K,\Fil^{-k} V / \Fil^{1-k} V). \]
By Lemma II.3.3 of \cite{Ber05}, we have $\Fil^j(V/\Fil^j V) = \{0\}$ for all $j \in \ZZ$. 
Therefore the quotient $\Fil^{-k} V / \Fil^{1-k} V$ has no subrepresentation all of whose Hodge--Tate weights are $\geq 1-k$. Corollary \ref{coro1} applies, so that $H^1_{\Iw,g}(K_\infty/K,\Fil^{-k} V / \Fil^{1-k} V)$ is a finite-dimensional $\Qp$-vector space, and therefore so is $H^1_{\Iw,g}(K_\infty/K, \Fil^{-k} V) / H^1_{\Iw,g}(K_\infty/K, \Fil^{1-k} V)$. 
Since $V = \Fil^{-k} V$ for $k \gg 0$, it follows by d\'evissage that 
$H^1_{\Iw,g}(K_\infty/K,V) / H^1_{\Iw,g}(K_\infty/K, \Fil^1 V)$
is a finite-dimensional $\Qp$-vector space. This proves Theorem \ref{theoB}.

\section{The twisted cyclotomic case}

The analogues of Theorem \ref{theoA} and Proposition \ref{prop1} still hold if one replaces the cyclotomic extension of $K$ by a twisted cyclotomic extension, namely (see \S 8 of \cite{Ber16}) if one replaces $K_{\cyc}$ by $K_{\cyc}^{\eta} = (K^{\alg})^{\ker \chi_{\cyc} \cdot \eta}$ with $\eta : G_K \to \Zp^\times$ an unramified character. The corresponding analogue of Theorem \ref{theoB} then also holds.

\begin{theo}
\label{theoC}
If $V$ is a de Rham representation of $G_K$ and if $K_\infty$ contains a twisted cyclotomic extension $K_{\cyc}^\eta$ of $K$, then $H^1_{\Iw,g}(K_\infty/K,V)  / H^1_{\Iw}(K_\infty/K,\Fil^1 V)$ is a finite-dimensional $\Qp$-vector space.
\end{theo}

\begin{proof}
The proof of this result is the same as the proof of Theorem \ref{theoB}, if we have an analogue of Proposition \ref{prop1} for $K_{\cyc}^{\eta}$. We now explain why such an analogue holds. If $P$ is the completion of the maximal unramified extension of $K$, then $P_{\cyc} = P^\eta_{\cyc}$. Take $V$ as in the statement of Proposition \ref{prop1}. By Lemma \ref{resinj}, the restriction-to-inertia map gives rise to an injective map $H^1_{\Iw}(K_{\cyc}^\eta/K,V) \to H^1_{\Iw}(P_{\cyc}/P,V)$. The space $H^1_{\Iw,g}(K_{\cyc}^\eta/K,V)$ therefore injects into $H^1_{\Iw,g}(P_{\cyc}/P,V)$, and Proposition \ref{prop1} implies that $H^1_{\Iw,g}(P_{\cyc}/P,V)$ is a finite-dimensional $\Qp$-vector space of dimension $\leq \dim(V)$. This now implies the corresponding analogue of Corollary \ref{coro1}, and the rest of the proof is then the same.
\end{proof}

\begin{lemm}
\label{resinj}
The restriction-to-inertia map $H^1_{\Iw}(K_{\cyc}^\eta/K,V) \to H^1_{\Iw}(P_{\cyc}/P,V)$ is injective.
\end{lemm}

\begin{proof}
If $\Lambda = \Zp \dcroc{\Gal(K_{\cyc}^\eta / K)}$ and $i \geq 0$, we have an isomorphism $H^i_{\Iw}(K_{\cyc}^\eta/K,V) = H^i(K,V \otimes_{\Zp} \Lambda)$ (see \S II.1 of \cite{Col98} for the argument in the cyclotomic case). We can replace $K$ by a subextension of $K_{\cyc}^\eta / K$ so that the map $\Gal(P_{\cyc}/P) \to \Gal(K_{\cyc}^\eta / K)$ is an isomorphism. The restriction-to-inertia map is then given by $H^1(G_K,V \otimes_{\Zp} \Lambda) \to H^1(I_K,V \otimes_{\Zp} \Lambda)$. By the inflation-restriction exact sequence, the kernel of this map is $H^1(G_K/I_K,  H^0(I_K, V \otimes_{\Zp} \Lambda) )$. Finally, $H^0(I_K, V \otimes_{\Zp} \Lambda) = H^0_{\Iw}(P_{\cyc}/P,V) = \{0\}$ (see ibid, or note that it is easy to see that $\projlim T^{G_{P_n}} = 0$).
\end{proof}

In particular, one can take for $K_\infty$ any extension containing the torsion points of a Lubin--Tate formal group. Indeed, if that group is attached to a uniformizer $\pi$ of a finite extension $F$ of $\Qp$ contained in $K$, and if $\chi_{\pi} : G_F \to \mathcal{O}_F^\times$ is the corresponding character, then $\Nm_{F/\Qp} (\chi_{\pi})$ is of the form $\chi_{\cyc} \cdot \eta$ as above.

\vspace{8pt}

\noindent\textbf{Acknowledgements.} I thank Rustam Steingart for his input, in particular for requesting the twisted cyclotomic case.

\vspace{8pt}

\noindent\textbf{AI disclosure.} This article was written by me. I used AI to (1) check and correct the mathematical content (2) look for references of technical results, in particular Lemmas \ref{lemm1} and \ref{resinj} (in both cases, it did not find a reference, but suggested proofs instead).

\bibliographystyle{smfalpha}
\bibliography{UnivNorms}

@article {Col98,
    AUTHOR = {Colmez, Pierre},
     TITLE = {Th\'{e}orie d'{I}wasawa des repr\'{e}sentations de de {R}ham
              d'un corps local},
   JOURNAL = {Ann. of Math. (2)},
  FJOURNAL = {Annals of Mathematics. Second Series},
    VOLUME = {148},
      YEAR = {1998},
    NUMBER = {2},
     PAGES = {485--571},
      ISSN = {0003-486X,1939-8980},
   MRCLASS = {11G45 (11R23 11S25 11S31 11S40 14F30 14G10)},
  MRNUMBER = {1668555},
MRREVIEWER = {Adolfo\ Quir\'{o}s},
       DOI = {10.2307/121003},
       URL = {https://doi.org/10.2307/121003},
}

@article {Ber05,
    AUTHOR = {Berger, Laurent},
     TITLE = {Repr\'{e}sentations de de {R}ham et normes universelles},
   JOURNAL = {Bull. Soc. Math. France},
  FJOURNAL = {Bulletin de la Soci\'{e}t\'{e} Math\'{e}matique de France},
    VOLUME = {133},
      YEAR = {2005},
    NUMBER = {4},
     PAGES = {601--618},
      ISSN = {0037-9484,2102-622X},
   MRCLASS = {11R23 (11F80 11S25 14F30)},
  MRNUMBER = {2233697},
MRREVIEWER = {Fr\'{e}d\'{e}ric\ D\'{e}glise},
       DOI = {10.24033/bsmf.2498},
       URL = {https://doi.org/10.24033/bsmf.2498},
}

@article {Ber16,
    AUTHOR = {Berger, Laurent},
     TITLE = {Multivariable {$(\varphi,\Gamma)$}-modules and locally
              analytic vectors},
   JOURNAL = {Duke Math. J.},
  FJOURNAL = {Duke Mathematical Journal},
    VOLUME = {165},
      YEAR = {2016},
    NUMBER = {18},
     PAGES = {3567--3595},
      ISSN = {0012-7094,1547-7398},
   MRCLASS = {14D24 (11F03 11S37)},
  MRNUMBER = {3577371},
MRREVIEWER = {Kimball\ L.\ Martin},
       DOI = {10.1215/00127094-3674441},
       URL = {https://doi.org/10.1215/00127094-3674441},
}

@article {PR00,
    AUTHOR = {Perrin-Riou, Bernadette},
     TITLE = {Repr\'{e}sentations {$p$}-adiques et normes universelles. {I}.
              {L}e cas cristallin},
   JOURNAL = {J. Amer. Math. Soc.},
  FJOURNAL = {Journal of the American Mathematical Society},
    VOLUME = {13},
      YEAR = {2000},
    NUMBER = {3},
     PAGES = {533--551},
      ISSN = {0894-0347,1088-6834},
   MRCLASS = {11S20 (11G25 11R23)},
  MRNUMBER = {1758753},
MRREVIEWER = {Mark\ Kisin},
       DOI = {10.1090/S0894-0347-00-00329-5},
       URL = {https://doi.org/10.1090/S0894-0347-00-00329-5},
}

@article {Pon25,
    AUTHOR = {Ponsinet, Gautier},
     TITLE = {Universal norms and the {F}argues-{F}ontaine curve},
   JOURNAL = {Math. Ann.},
  FJOURNAL = {Mathematische Annalen},
    VOLUME = {392},
      YEAR = {2025},
    NUMBER = {2},
     PAGES = {2853--2912},
      ISSN = {0025-5831,1432-1807},
   MRCLASS = {11R23 (11F80 11F85 14G45)},
  MRNUMBER = {4906335},
MRREVIEWER = {K\^{a}z\i m\ B\"{u}y\"{u}kboduk},
       DOI = {10.1007/s00208-025-03131-8},
       URL = {https://doi.org/10.1007/s00208-025-03131-8},
}

@unpublished{Pon24,
   AUTHOR = {Ponsinet, Gautier},
   TITLE = {{B}loch--{K}ato groups over perfectoid fields and {G}alois theory of $p$-adic periods},
  NOTE = {Preprint},
  YEAR = {2024},}
\end{document}